\documentclass[11pt]{amsart}

\usepackage{amssymb, amsthm, amsmath, gensymb, dsfont}
\usepackage{graphicx, comment}
\usepackage[small]{caption}
\usepackage{subcaption}
\usepackage{epsfig}
\usepackage{tikz, pgfplots, float}

\usepackage{amsfonts}

\usepackage[utf8]{inputenc}
\usepackage{fullpage}
\usepackage{framed}

\usepackage{enumerate}
\usepackage{url}
\usepackage[breaklinks]{hyperref}
\usepackage{cleveref}
\hypersetup{
	colorlinks = true, 
	urlcolor = cyan, 
	linkcolor = teal, 
	citecolor = cyan 
}

\newtheorem{theorem}{Theorem}
\newtheorem{definition}[theorem]{Definition}
\newtheorem{lemma}[theorem]{Lemma}
\newtheorem{claim}[theorem]{Claim}
\crefname{claim}{claim}{claims}

\newcommand{\eps}{\varepsilon}

\newcommand{\NN}{\mathbb{N}} 
\newcommand{\RR}{\mathbb{R}} 

\usepackage{tikz}
\tikzstyle{vtx} = [circle, fill, inner sep=0.7]

\title{On the number of regular simplices in odd dimensions}

\author{Felix Christian Clemen}
\address{Felix Christian Clemen \newline University of Victoria, Canada}
\email{fclemen@uvic.ca}

\author{Dingyuan Liu}
\address{Dingyuan Liu \newline Karlsruhe Institute of Technology, Germany}
\email{liu@mathe.berlin}

\date{}

\begin{document}

\begin{abstract}
Let $S^k_d(n)$ denote the maximum number of regular $(k-1)$-simplices spanned by $n$ points in $\mathbb{R}^d$. For all fixed $r\geq k\geq3$, the exact value of $S^k_{2r}(n)$ was recently determined by Dumitrescu and the authors for all sufficiently large $n$ when $k=3$, and conditionally on an optimization problem when $k\geq4$. In this paper we establish a sharp estimate in odd dimensions by proving that
\[S^k_{2r+1}(n)=\binom{r}{k}\left(\frac{n}{r}\right)^k+\Theta(n^{k-2/3})\]
for all fixed $r\geq k\geq3$. This can be viewed as a generalization of the result by Erd\H{o}s and Pach on unit distances in odd dimensions. We also establish a structural characterization of nearly extremal point sets. A notable difference from the even-dimensional case is that nearly extremal point sets in odd dimensions admit two distinct configurations. The proof leverages techniques from hypergraph Tur\'an theory and linear algebra, with additional geometric arguments.
\end{abstract}

\maketitle

\section{Introduction}
\label{sec:intro}
Combinatorial geometry studies extremal problems concerning geometric objects. The origins of the field can be traced back to the celebrated unit-distance problem of Erd\H{o}s~\cite{Er46,Er59}, which asks for the maximum number of pairs of points at unit distance among $n$ points in $\RR^d$. See~\cite{Er67,EP90} for an asymptotic solution in the case $d\geq4$, and see~\cite{NA26,AI26,WS26} for a recent breakthrough in dimension $2$. Shortly after introducing the unit-distance problem, Erd\H{o}s and Purdy~\cite{Er75,EP71,EP75,EP76} considered a natural generalization, that is, to maximize the number of regular $(k-1)$-simplices spanned by $n$ points in $\RR^d$. Here, a regular $(k-1)$-simplex is a set of $k$ points that are pairwise equidistant.

Define the extremal function $S^k_d(n)$ as the maximum number of regular $(k-1)$-simplices spanned by $n$ points in $\RR^d$. Erd\H{o}s~\cite{Er94} specifically conjectured that $S^3_6(n)=n^3/27-o(n^3)$. This conjecture was recently confirmed by Dumitrescu and the authors as part of the following more general result.

\begin{theorem}[Clemen--Dumitrescu--Liu~\cite{CDL26}]
\label{simplex}
For any fixed integers $r\geq k\geq3$,
\[S^k_{2r}(n)=\binom{r}{k}\left(\frac{n}{r}\right)^k+\Theta(n^{k-1})\]
and
\[\binom{r}{k}\left(\frac{n}{r}\right)^k+\Omega(n^{k-2/3})\leq S^k_{2r+1}(n)\leq\binom{r}{k}\left(\frac{n}{r}\right)^k+o(n^k).\]
\end{theorem}

In fact, the exact value of $S^k_{2r}(n)$ was determined in~\cite{CDL26} for sufficiently large $n$, conditional on the solution of a certain optimization problem. The key ingredient was a stability theorem showing that every set of $n$ points in $\RR^{2r}$ spanning nearly the maximum number of regular $(k-1)$-simplices must exhibit a highly structured configuration. However, the argument in~\cite{CDL26} relied crucially on the restriction to even dimensions, leaving the order of magnitude of the lower-order term in $S^k_{2r+1}(n)$ undetermined. The main contribution of this paper is to close this gap by establishing the following sharper estimate for $S^k_{2r+1}(n)$.

\begin{theorem}
\label{simplex_odd}
Let $r\geq k\geq3$ be fixed integers. Then \[S^k_{2r+1}(n)=\binom{r}{k}\left(\frac{n}{r}\right)^k+\Theta(n^{k-2/3}).\]
\end{theorem}

In the case $k=2$, Erd\H{o}s and Pach~\cite{EP90} proved that, for every $r\geq2$, the maximum number of pairs at unit distance among $n$ points in $\RR^{2r+1}$ is $\binom{r}{2}(n/r)^2+\Theta(n^{4/3})$. From this perspective,~\Cref{simplex_odd} can be viewed as a ``hypergraph'' extension of their result, although the passage to higher uniformity requires a much more delicate analysis.

We remark that the principal bottleneck in determining the leading constant in $\Theta(n^{k-2/3})$ is the unit-distance problem for points lying on a sphere in $\RR^3$. Indeed, our proof of~\Cref{simplex_odd} shows that every extremal point set must have a positive proportion of its points lying on such a sphere. Consequently, obtaining an asymptotically sharp count of the regular $(k-1)$-simplices containing two points from this spherical subset would require an asymptotic solution to the corresponding unit-distance problem. However, the best known results on this problem determine only the order of magnitude~\cite{CEGSW90,EHP89}, and the leading constants in the lower and upper bounds remain far apart. Therefore, any further refinement of~\Cref{simplex_odd} currently appears to be out of reach.

Let us quickly compare the proofs of~\Cref{simplex,simplex_odd}. The proof of~\Cref{simplex} in~\cite{CDL26} started by exploiting linear algebraic constraints to identify forbidden configurations, and then applied methods from hypergraph Tur\'an theory to bound $S^k_{2r}(n)$. However, many of the algebraic constraints exploited there are inherently even-dimensional and thus fail in odd dimensions. To address this issue, our proof of~\Cref{simplex_odd} follows the same framework but replaces these missing algebraic constraints with geometric observations. This constitutes the main new ingredient of our proof.

The key to obtaining a sharper estimate for $S^k_{2r+1}(n)$ is to establish an odd-dimensional version of the stability theorem, which captures the structure of a nearly extremal point set in $\RR^{2r+1}$. It was shown in~\cite{CDL26} that every nearly extremal point set in $\RR^{2r}$ must be largely distributed among $r$ pairwise orthogonal circles with a common center and radius. However, the configuration of such point sets in odd dimensions tends to be more intricate. The difference arises from the fact that $\RR^{2r+1}$ cannot be spanned by $r$ circles with a common center. In particular, considerable effort is required to show that the additional dimension leads to only two configurations: Informally, either one of the $r$ circles expands into a $2$-dimensional sphere, or the $r$ circles are arranged in a skewed manner along the remaining free dimension.

To state our stability result, we introduce some terminology. We give only brief definitions here to avoid interrupting the flow of the presentation. Formal definitions of affine spaces, spheres, and orthogonality are provided in~\Cref{definitions}. The dimension of a sphere $C$, written as $\dim{C}$, is defined to be the dimension of its affine hull minus $1$. For example, a circle is a $1$-dimensional sphere. Two spheres are said to be orthogonal if the affine spaces spanned by them are orthogonal. The stability result can then be stated as follows.

\begin{theorem}
\label{simplex_stability}
Let $r\geq k\geq3$ be fixed integers. For every $\eps>0$ there exist $\delta>0$ and an integer $n_0>0$ such that the following holds for all integers $n\geq n_0$: Suppose that $X\subseteq\RR^{2r+1}$ is a set of $n$ points spanning at least $S^k_{2r+1}(n)-\delta n^k$ regular $(k-1)$-simplices. Then there exist a partition $X=A_0\dot\cup A_1\dot\cup\dots\dot\cup A_r$ with $|A_0|<\eps n$ and $|A_i|>n/r-\eps n$ for each $i\in[r]$, and pairwise orthogonal spheres $C_1,\dots,C_r\subseteq\RR^{2r+1}$ with $A_i\subseteq C_i$ for every $i\in[r]$, satisfying the following property. If $r\geq4$, then
\begin{enumerate}
    \item $\dim{C_1}=\dots=\dim{C_{r-1}}=1$ and $\dim{C_r}=2$;
    \item $C_1,\dots,C_r$ have the same center and the same radius.
\end{enumerate}
If $r=3$ (and hence $k=3$), then either the above holds, or
\begin{enumerate}
    \item[(1')] $C_1,C_2,C_3$ are circles with pairwise distinct centers;
    \item[(2')] the distance between any two points on different circles is the same.
\end{enumerate}
\end{theorem}

Note that both nearly extremal configurations described in~\Cref{simplex_stability} can occur. The existence of the first is immediate. For the second, consider two unit circles $C_1$ and $C_3$ lying in $\mathbb{R}^2\times\{0\}^4\times\{1\}$ and $\{0\}^4\times\mathbb{R}^2\times\{3\}$, centered at $(0,0,0,0,0,0,1)$ and $(0,0,0,0,0,0,3)$, respectively. Then, let $C_2$ be a circle of radius $2$ lying in $\{0\}^2\times\mathbb{R}^2\times\{0\}^2\times\{2\}$ centered at $(0,0,0,0,0,0,2)$. In this configuration every triangle consisting of three points from different circles is equilateral.

The paper is organized as follows. In~\Cref{sec:preliminary}, we give necessary definitions and collect auxiliary lemmas. In~\Cref{sec:stability} we prove~\Cref{simplex_stability}, and in~\Cref{sec:number} we establish~\Cref{simplex_odd}.

\section{Preliminaries}
\label{sec:preliminary}
\subsection{Affine spaces, spheres, and orthogonality}
\label{definitions}
A non-empty subset $\mathcal{A}\subseteq\mathbb{R}^d$ is called an affine space if there exist an element $x\in\mathbb{R}^d$ and a linear subspace $\mathcal{S}\subseteq\mathbb{R}^d$ such that
\[\mathcal{A}=x+\mathcal{S}:=\{x+s:\,s\in\mathcal{S}\}.\]
The linear subspace $\mathcal{S}$ is referred to as the associated linear space of $\mathcal{A}$. Note that the associated linear space of an affine space is uniquely determined.

The dimension of an affine space $\mathcal{A}\subseteq\mathbb{R}^d$, written $\dim\mathcal{A}$, is defined as the dimension of the linear space associated with $\mathcal{A}$. Two affine spaces are called orthogonal whenever their associated linear spaces are orthogonal. Consequently, if $\mathcal{A}_1,\dots,\mathcal{A}_r\subseteq\mathbb{R}^d$ are pairwise orthogonal affine spaces, then $\sum_{i\in[r]}\dim\mathcal{A}_i\leq \dim\RR^d=d$.

For a nonempty subset $A\subseteq\RR^d$, we write $\mathrm{Aff}(A)$ for the affine space generated by $A$. That is,
\[\mathrm{Aff}(A):=a+\mathrm{Span}(A-a)\quad\text{for any $a\in A$,}\]
where $\mathrm{Span}(A-a)$ denotes the linear subspace of $\RR^d$ spanned by $A-a$.

Given $d\in\mathbb{N}$, a $(d-1)$-dimensional sphere is a set of the form $\{x\in\mathcal{A}:\,||x-c||=\gamma\}$, where $\mathcal{A}$ is a $d$-dimensional affine space, $c\in\mathcal{A}$, and $\gamma>0$. The point $c$ and the number $\gamma$ are called the center and the radius of the sphere, respectively. In particular, a $0$-dimensional sphere consists of two points, while a $1$-dimensional sphere is a circle. Two spheres are said to be orthogonal if the affine spaces spanned by them are orthogonal.

The following lemma reveals a connection between forcing regular simplices and the orthogonality of affine spaces.
\begin{lemma}[{\cite[Lemma~15]{CDL26}}]
\label{witzig}
Let $d\geq1$ and $k\geq3$ be integers. Let $A_1,\dots,A_k\subseteq \RR^d$ with $\lvert{A_i}\rvert\geq3$ for each $i\in[k]$. If every $k$-tuple $(a_1,\dots,a_k)\in \prod_{i\in[k]}A_i$ forms a regular $(k-1)$-simplex, then $\mathrm{Aff}(A_1),\dots,\mathrm{Aff}(A_k)$ are pairwise orthogonal affine spaces of dimension at least $2$. Furthermore, $A_1,\dots,A_k$ are subsets of $k$ spheres in $\mathrm{Aff}(A_1),\dots,\mathrm{Aff}(A_k)$, respectively.
\end{lemma}

The next lemma concerns the geometric properties of a collection of pairwise orthogonal spheres.
\begin{lemma}[{\cite[Lemma~16]{CDL26}}]
\label{spheres_distance_center}
Let $d\geq1$ and $k\geq3$ be integers. Let $C_1,\dots,C_k\subseteq\mathbb{R}^d$ be a collection of pairwise orthogonal spheres. For each $i\in[k]$ let $A_i\subseteq C_i$ be such that $\mathrm{Aff}(A_i)=\mathrm{Aff}(C_i)$.
\begin{enumerate}[(i)]
    \item If every $k$-tuple in $\prod_{i\in[k]}A_i$ forms a regular $(k-1)$-simplex, then every $k$-tuple in $\prod_{i\in[k]}C_i$ forms a regular $(k-1)$-simplex.
    \item If every $k$-tuple in $\prod_{i\in[k]}C_i$ forms a regular $(k-1)$-simplex and $\sum_{i\in[k]}\dim\mathrm{Aff}(C_i)=d$, then $C_1,\dots,C_k$ have the same center.
\end{enumerate}
\end{lemma}

\subsection{Results from extremal graph theory}
Given an integer $k\geq2$, a $k$-graph $H$ is an ordered pair $(V(H),E(H))$, where $V(H)$ is the set of vertices and $E(H)\subseteq\binom{V(H)}{k}$ is the set of edges. Let $v(H)$ and $e(H)$ denote the numbers of vertices and edges in $H$, respectively.

A $k$-graph $H$ is said to be $r$-partite for some integer $r\geq k$ if there exists a partition $V(H)=V_1\dot\cup V_2\dot\cup\dots\dot\cup V_r$, such that every edge in $H$ has at most one vertex from each part $V_i$. 

For a $k$-graph $H$, the extremal number $\mathrm{ex}(n,H)$ is defined as the largest number of edges in an $n$-vertex $k$-graph containing no copy of $H$. Given an integer $t\geq1$, the $t$-blowup of $H$, denoted by $H(t)$, is a $k$-graph obtained from $H$ by replacing every vertex $v\in V(H)$ with an independent set $I_v$ of size $t$ and replacing every edge $e=\{v_1,\dots,v_k\}\in E(H)$ with a complete $k$-partite $k$-graph spanned
by $I_{v_1}\dot\cup\dots\dot\cup I_{v_k}$.

\begin{definition}
Given integers $r\geq k\geq3$, $H^{(k)}_{r+1}$ is defined as a $k$-graph obtained from a complete graph $K_{r+1}$ by enlarging each edge with $k-2$ new vertices, where all newly added vertices are distinct.
\end{definition}

In what follows, we collect several results proved in~\cite{CDL26}. These results will be used in the proof of~\Cref{simplex_stability} to extract the underlying structure of the point set.

\begin{lemma}[{\cite[Lemma~10]{CDL26}}]
\label{new:stability}
Let $r\geq k\geq3$ and $t\geq1$ be fixed integers. For any $\eps>0$ there exists $\delta>0$ such that the following holds. Let $G$ be an $n$-vertex $k$-graph that is $H^{(k)}_{r+1}(t)$-free. If $G$ has at least $\binom{r}{k}\left(\frac{n}{r}\right)^k-\delta n^k$ edges, then $G$ contains an $r$-partite $k$-subgraph with at least $\binom{r}{k}(n/r)^k-\eps n^k$ edges.
\end{lemma}

\begin{lemma}[{\cite[Lemma~17]{CDL26}}]
\label{Hfree}
Let $r\geq k\geq3$ be integers. Let $G$ be a $k$-graph, whose vertices are points in $\mathbb{R}^{2r+1}$ and whose edges are the regular $(k-1)$-simplices spanned by its vertices. Then $G$ is $H^{(k)}_{r+1}(3)$-free.
\end{lemma}

The following lemma has been widely used in the literature (see, e.g.,~\cite[Lemma~13]{CDL26} for a proof).
\begin{lemma}
\label{parts}
Let $r\geq k\geq2$ be fixed integers. For any $\eps>0$ there exists $\delta>0$ such that the following holds. If an $n$-vertex $r$-partite $k$-graph has at least $\binom{r}{k}(n/r)^k-\delta n^k$ edges, then each part must be of size at least $n/r-\eps n$.
\end{lemma}

\section{The structure of a nearly extremal point set}
\label{sec:stability}
In this section we prove our stability result, i.e.,~\Cref{simplex_stability}.
\begin{proof}[Proof of~\Cref{simplex_stability}]
Let $r\geq k\geq3$ be fixed integers. Let $\eps>0$ be given. We choose the parameters $\eps_0,\eps_1,\eps_2,\delta>0$ according to the hierarchy
\[0<\delta \ll \eps_2 \ll \eps_1 \ll \eps_0 \ll \eps,\]
where the parameters are chosen successively from right to left, sufficiently small such that the hypotheses of the relevant lemmas are satisfied. In addition, assume without loss of generality that
\[\eps<(1000)^{-k}r^{-1000k}.\]
Since $S^k_{2r+1}(n)=\binom{r}{k}(n/r)^k+o(n^k)$ by~\Cref{simplex}, there exists some $n_0\in\NN$ such that $S^k_{2r+1}(n)-\delta n^k>\binom{r}{k}(n/r)^k-\eps_2n^k$ holds for all $n\geq n_0$. Suppose $n\geq n_0$ throughout the proof.

Let $X\subseteq\RR^{2r+1}$ be a set of $n$ points that spans at least $S^k_{2r+1}(n)-\delta n^k$ regular $(k-1)$-simplices. Let $G$ be a $k$-graph, whose vertex set is $X$ and whose edges correspond to regular $(k-1)$-simplices spanned by $X$.

By~\Cref{Hfree}, $G$ is $H^{(k)}_{r+1}(3)$-free. Since $G$ has at least
\[S^k_{2r+1}(n)-\delta n^k>\binom{r}{k}(n/r)^k-\eps_2n^k\]
edges,~\Cref{new:stability} guarantees that $G$ contains an $r$-partite $k$-subgraph $M$ with parts $V_1\dot\cup\dots\dot\cup V_r$ and at least $\binom{r}{k}(n/r)^k-\eps_1n^k$ edges. Moreover, by~\Cref{parts} we have $|V_i|\geq n/r-\eps_0n$ for all $i\in[r]$. A missing edge in $M$ is called transversal if it contains at most one vertex from each part $V_i$. By Maclaurin's inequality~\cite{BP03}, every $n$-vertex $r$-partite $k$-graph has at most $\binom{r}{k}(n/r)^k$ edges. It follows that $M$ has at most $\eps_1n^k$ missing transversal edges.

\begin{claim}
\label{lustig}
There exists $A_i\subseteq V_i$ for each $i\in[r]$, such that
\begin{enumerate}[(i)]
    \item $|A_i|>n/r-(3r)^k\eps  n$ and $\mathrm{Aff}(A_i)$ is an affine space with $2\leq\dim\mathrm{Aff}(A_i)\leq3$;
    \item there are spheres $C_1,\dots,C_r$, such that $A_i\subseteq C_i \subseteq\mathrm{Aff}(A_i)$ for all $i\in[r]$;
    \item $\mathrm{Aff}(A_i)=\mathrm{Aff}(A_i')$ for any $A_i'\subseteq A_i$ of size at least $|A_i|-\eps n$.
\end{enumerate}
\end{claim}
\begin{proof}[Proof of~\Cref{lustig}]
We shall prove the statement for $A_r$, the cases of $A_1,\dots,A_{r-1}$ admit the same proof. For any distinct $\ell_1,\dots,\ell_{k-1}\in[r-1]$, a $(k-1)$-tuple $(v_{\ell_1},\dots,v_{\ell_{k-1}})\in \prod_{i\in[k-1]}V_{\ell_i}$ is called ``good'' if there are at least $|V_r|-\eps n$ vertices $v_r\in V_r$ such that $v_{\ell_1}\dots v_{\ell_{k-1}}v_r$ is an edge in $M$, otherwise it is ``bad''. The total number of bad tuples is at most $\eps n^{k-1}$, as otherwise the number of missing transversal edges in $M$ would be at least \[\eps n^{k-1}\cdot\eps n=\eps^2n^k>\eps_1n^k,\] a contradiction.
Thus, if we take a $(k-1)$-tuple $(v_{\ell_1},\dots,v_{\ell_{k-1}})\in \prod_{i\in[k-1]}V_{\ell_i}$ uniformly at random, the probability that $(v_{\ell_1},\dots,v_{\ell_{k-1}})$ is bad is at most
\[\mathbb{P}_{\text{bad}}\leq\frac{\eps n^{k-1}}{\prod_{i\in[k-1]}|V_{\ell_i}|}\leq\frac{\eps n^{k-1}}{(n/r-\eps_0n)^{k-1}}<r^{-2k}.\]
Now for each $i\in[r-1]$ we take a subset $V_i'\subseteq V_i$ of size $3$ uniformly at random. Then by the union bound
\[\mathbb{P}\left(\forall\,\mathcal{I}\in\binom{[r-1]}{k-1}:\text{ all the $(k-1)$-tuples in $\prod_{i\in\mathcal{I}}V_i'$ are good}\right)\geq 1- \binom{r-1}{k-1}3^{k-1}\mathbb{P}_{\text{bad}}>0.\]
Accordingly, there exists $V_i'\subseteq V_i$ of size $3$ for each $i\in[r-1]$, such that all $(k-1)$-tuples among them are good. Let $A_r\subseteq V_r$ be the common neighborhood of the $(k-1)$-tuples in $\bigcup_{\mathcal{I}\in\binom{[r-1]}{k-1}}\prod_{i\in\mathcal{I}}V_i'$. Then
\begin{align*}
|A_r|\geq |V_r|- \binom{r-1}{k-1}3^{k-1}\eps n\geq\frac{n}{r}-\eps_0n-(3r)^{k-1}\eps n>\frac{n}{r}-(3r)^k\eps n + \eps n.
\end{align*}
In particular, for any $\mathcal{I}\in\binom{[r-1]}{k-1}$, every $k$-tuple in $\left(\prod_{i\in\mathcal{I}}V_i'\right)\times A_r$ appears as an edge in $M$, i.e., a regular $(k-1)$-simplex. By~\Cref{witzig} we have that $\mathrm{Aff}(V_1'),\dots,\mathrm{Aff}(V_{r-1}'),\mathrm{Aff}(A_r)$ are pairwise orthogonal affine spaces of dimension at least $2$. Since \[\dim\mathrm{Aff}(V_1')+\dots+\dim\mathrm{Aff}(V_{r-1}')+\dim\mathrm{Aff}(A_r)\leq\dim\RR^{2r+1}=2r+1,\]
we have $\dim\mathrm{Aff}(A_r)\leq3$. Moreover, by~\Cref{witzig} the points of $A_r$ all lie on a sphere in $\mathrm{Aff}(A_r)$. 

To show that $A_r$ satisfies item (iii), we distinguish two cases. If $\dim\mathrm{Aff}(A_r)=2$, as the points in $A_r$ are non-collinear, any subset $A_r'\subseteq A_r$ of size at least $3$ satisfies $\dim\mathrm{Aff}(A_r')=2=\dim\mathrm{Aff}(A_r)$. If $\dim\mathrm{Aff}(A_r)=3$ and there is some $A_r'\subseteq A_r$ of size at least $|A_r|-\eps n$, such that $\dim\mathrm{Aff}(A_r')=2<\dim\mathrm{Aff}(A_r)$, then we can replace $A_r$ by $A_r'$ and return to the first case. Since $|A_r'|\geq|A_r|-\eps n\geq n/r-(3r)^k\eps n$, the new $A_r$ still satisfies items (i) and (ii). This completes the proof of~\Cref{lustig}.
\end{proof}

\begin{claim}
\label{sahne}
There exists at most one $i\in[r]$, such that $\dim\mathrm{Aff}(A_i)=3$.
\end{claim}
\begin{proof}[Proof of~\Cref{sahne}]
Suppose towards a contradiction that at least two of $\mathrm{Aff}(A_1),\dots,\mathrm{Aff}(A_r)$ are $3$-dimensional. Then the affine spaces $\mathrm{Aff}(A_1),\dots,\mathrm{Aff}(A_{r})$ cannot be pairwise orthogonal, because otherwise
\[\dim\RR^{2r+1}\geq\sum_{i=1}^r\dim\mathrm{Aff}(A_i)\geq2r+2,\]
a contradiction. Assume without loss of generality that $\mathrm{Aff}(A_1)$ and $\mathrm{Aff}(A_2)$ are not orthogonal. There is a sphere $C_2$ given by~\Cref{lustig} item (ii), such that $A_2\subseteq C_2\subseteq\mathrm{Aff}(A_2)$. Let $A_1'\subseteq A_1$ be the set of points in $A_1$ whose projections onto $\mathrm{Aff}(A_2)$ are not the center of $C_2$. Since $\mathrm{Aff}(A_1)$ and $\mathrm{Aff}(A_2)$ are not orthogonal, we have $\dim\mathrm{Aff}(A_1\setminus A_1')<\dim\mathrm{Aff}(A_1)$. Then by~\Cref{lustig} item (iii), $|A_1\setminus A_1'|<|A_1|-\eps n$, namely, $|A_1'|>\eps n$. Moreover, every point $v_1\in A_1'$ is equidistant to at most $|A_2|-\eps n$ points in $A_2$, as otherwise the projection of $v_1$ onto $\mathrm{Aff}(A_2)$ would be the center of a sphere spanned by more than $|A_2|-\eps n$ points in $A_2$ and by~\Cref{lustig} item (iii) this sphere would be exactly $C_2$, a contradiction. Then, for each $v_1\in A_1'$ and each $(v_3,\dots,v_k)\in\prod_{i=3}^k V_i$, there are at least $\eps n$ vertices $v_2\in A_2\subseteq V_2$ such that $(v_1,\dots,v_k)$ is a missing transversal edge in $M$. Therefore, the number of missing transversal edges in $M$ is at least
\[|A_1'|\cdot\eps n\cdot\prod_{i=3}^{k}|V_i|\geq\eps n\cdot\eps n\cdot\left(\frac{n}{r}-\eps_0n\right)^{k-2}>\eps^2r^{-k}n^k>\eps_1n^k,\]
a contradiction.
\end{proof}

Assume without loss of generality $\dim\mathrm{Aff}(A_1)\leq\dim\mathrm{Aff}(A_2)\leq\dots\leq\dim\mathrm{Aff}(A_r)$.
In particular, by~\Cref{sahne} we have $\dim\mathrm{Aff}(A_1)=\dots=\dim\mathrm{Aff}(A_{r-1})=2$.

\begin{claim}
\label{core}
There exist subsets $A_i'\subseteq A_i$ for each $i\in[r]$, such that
\begin{enumerate}[(i)]
    \item $|A_1'|=\dots=|A_{r-1}'|=3$ and $|A_r'|>|A_r|-\eps n$;
    \item every $k$-tuple in $\bigcup_{\mathcal{I}\in\binom{[r-1]}{k-1}}A_r'\times\prod_{i\in\mathcal{I}}A_i'$ forms a regular $(k-1)$-simplex;
    \item $\mathrm{Aff}(A_i')=\mathrm{Aff}(A_i)$ holds for each $i\in[r]$.
\end{enumerate}
\end{claim}
\begin{proof}[Proof of~\Cref{core}]
We repeat the proof of~\Cref{lustig} on $A_1,\dots,A_r$ with $\eps'=\eps/(3r)^k$ in place of $\eps$, which yields that with positive probability there exists $A_i'\subseteq A_i$ of size $3$ for each $i\in[r-1]$, such that every $(k-1)$-tuple in $\bigcup_{\mathcal{I}\in\binom{[r-1]}{k-1}}\prod_{i\in\mathcal{I}}A_i'$ forms edges with at least $|A_r|-\eps'n$ vertices in $A_r$. Let $A_r'\subseteq A_r$ be the common neighborhood of $\bigcup_{\mathcal{I}\in\binom{[r-1]}{k-1}}\prod_{i\in\mathcal{I}}A_i'$. Note that $|A_r'|>|A_r|-(3r)^k\eps'n=|A_r|-\eps n$. It is immediate that the chosen $A_1',\dots,A_r'$ satisfy items (i) and (ii).

To show that $A_1',\dots,A_r'$ satisfy item (iii), recall that $\dim\mathrm{Aff}(A_1)=\dots=\dim\mathrm{Aff}(A_{r-1})=2$ and $A_1,\dots,A_{r-1}$ are spherical. Namely, we have $\dim\mathrm{Aff}(A_i')=2=\dim\mathrm{Aff}(A_i)$ for all $i\in[r-1]$, which implies that $\mathrm{Aff}(A_i')=\mathrm{Aff}(A_i)$ for all $i\in[r-1]$. Moreover, since $|A_r'|>|A_r|-\eps n$, by~\Cref{lustig} item (iii) we have $\mathrm{Aff}(A_r')=\mathrm{Aff}(A_r)$.
\end{proof}

\begin{claim}
\label{creme}
$C_1,\dots,C_r$ are pairwise orthogonal spheres.
\end{claim}
\begin{proof}[Proof of~\Cref{creme}]
By~\Cref{core} item (ii) and~\Cref{witzig}, we have that $\mathrm{Aff}(A_1'),\dots,\mathrm{Aff}(A_r')$ are pairwise orthogonal. It then follows from~\Cref{core} item (iii) that $\mathrm{Aff}(A_1),\dots,\mathrm{Aff}(A_r)$ are pairwise orthogonal, namely, $C_1,\dots,C_r$ are pairwise orthogonal spheres.  
\end{proof}

\begin{claim}
\label{spheres}
Any two points lying on different spheres among $C_1,\dots,C_r$ are at the same distance.
\end{claim}
\begin{proof}[Proof of~\Cref{spheres}]
From~\Cref{core} item (iii) it follows that $\mathrm{Aff}(A_i')=\mathrm{Aff}(C_i)$ for every $i\in[r]$. Moreover, for any $\mathcal{I}\in\binom{[r-1]}{k-1}$, by~\Cref{core} item (ii) we have that every $k$-tuple in $A_r'\times\prod_{i\in\mathcal{I}}A_i'$ forms a regular $(k-1)$-simplex. Then by~\Cref{spheres_distance_center} item (i), every $k$-tuple in $C_r\times\prod_{i\in\mathcal{I}}C_i$ forms a regular $(k-1)$-simplex. This further implies that all the inter-sphere distances among $C_1,\dots,C_r$ are equal.
\end{proof}

Our next goal is to obtain information about the centers of $C_1,\dots,C_r$. We shall distinguish two cases depending on the value of $\sum_{i\in[r]}\dim\mathrm{Aff}(C_i)$.
\begin{claim}
\label{good_centers}
If $\sum_{i\in[r]}\dim\mathrm{Aff}(C_i)=2r+1$, then $C_1,\dots,C_r$ all have the same center.
\end{claim}
\begin{proof}[Proof of~\Cref{good_centers}]
The claim follows by combining~\Cref{spheres} and~\Cref{spheres_distance_center} item (ii).
\end{proof}

\begin{claim}
\label{bad_centers}
If $\sum_{i\in[r]}\dim\mathrm{Aff}(C_i)=2r$ and $r\geq4$, then at least $r-1$ among $C_1,\dots,C_r$ have the same center.
\end{claim}
\begin{proof}[Proof of~\Cref{bad_centers}]
Let $c_1,\dots,c_r$ denote the centers of $C_1,\dots,C_r$, respectively. The assumption $\sum_{i\in[r]}\dim\mathrm{Aff}(C_i)=2r$ implies that $\dim\mathrm{Aff}(C_i)=2$ for all $i\in[r]$.

Since $\mathrm{Aff}(C_1),\dots,\mathrm{Aff}(C_r)$ are pairwise orthogonal by~\Cref{creme} and their dimensions sum to $2r$, we can equip $\RR^{2r+1}$ with an orthonormal basis given by the union of $B_1,\dots,B_r$, where each $B_i$ is an orthonormal basis of the linear subspace $\mathrm{Aff}(C_i)-c_i$, together with one additional basis vector. Under this basis of $\RR^{2r+1}$, we can assume that
\[\mathrm{Aff}(C_i)=c_i+\underbrace{\{0\}\times\dots\times\{0\}}_{2(i-1)\text{ times}}\times\RR^{2}\times\underbrace{\{0\}\times\dots\times\{0\}}_{2r+1-2i\text{ times}}\]
for each $i\in[r]$. Moreover, since $C_1,\dots,C_r$ are pairwise orthogonal by~\Cref{creme} and their inter-sphere distances are all equal by~\Cref{spheres}, we have that, for any distinct $i,j\in[r]$, the projection of $c_i$ onto $\mathrm{Aff}(C_j)$ is $c_j$, i.e., $c_i$ has the same $(2j-1)$-th and $(2j)$-th coordinates as $c_j$. Consequently, the centers $c_1,\dots,c_r$ of $C_1,\dots,C_r$ coincide in their first $2r$ coordinates.

Let $\gamma_i$ denote the radius of each $C_i$.

Suppose that the last coordinates of $c_1,\dots,c_r$ take more than $3$ values. Assume without loss of generality that the last coordinates of $c_1,\dots,c_4$ are $a_1,\dots,a_4$, respectively, where $a_1>a_2>a_3>a_4$. For any $x_1\in C_1,\dots,x_4\in C_4$, by~\Cref{spheres} we have
\[\gamma_1^2+(a_1-a_2)^2+\gamma_2^2=||x_1-x_2||^2=||x_1-x_3||^2=\gamma_1^2+(a_1-a_3)^2+\gamma_3^2,\]
which implies that $\gamma_2>\gamma_3$. On the other hand,
\[\gamma_2^2+(a_2-a_4)^2+\gamma_4^2=||x_2-x_4||^2=||x_3-x_4||^2=\gamma_3^2+(a_3-a_4)^2+\gamma_4^2,\]
yielding that $\gamma_2<\gamma_3$, a contradiction.

Next, suppose there are $3$ possible values for the last coordinates. We again assume the last coordinates of $c_1,c_2,c_3$ are $a_1,a_2,a_3$, respectively, where $a_1>a_2>a_3$. Since $r\geq4$, we have $c_4=c_i$ for some $i\in[3]$. Let $\{j,\ell\}=[3]\setminus\{i\}$ and $x_1\in C_1,\dots,x_4\in C_4$. By~\Cref{spheres} we have
\[\gamma_i^2+(a_i-a_j)^2+\gamma_j^2=||x_i-x_j||^2=||x_4-x_j||^2=\gamma_4^2+(a_i-a_j)^2+\gamma_j^2,\]
namely, $\gamma_i=\gamma_4$. Then, as $2\gamma_i^2=||x_i-x_4||^2=||x_i-x_j||^2$, we obtain $\gamma_i^2=(a_i-a_j)^2+\gamma_j^2$. Similarly, we obtain $\gamma_i^2=(a_i-a_\ell)^2+\gamma_\ell^2$ and thus
\[\gamma_j^2+(a_j-a_\ell)^2+\gamma_\ell^2=||x_j-x_\ell||^2=||x_i-x_4||^2=2\gamma_i^2=\gamma_j^2+(a_i-a_j)^2+(a_i-a_\ell)^2+\gamma_\ell^2,\]
implying that $(a_j-a_\ell)^2=(a_i-a_j)^2+(a_i-a_\ell)^2$, a contradiction.

Hence, there are at most $2$ possible values for the last coordinates, say $a_1$ and $a_2$ with $a_1>a_2$. It remains to exclude the case where two centers, say $c_1$ and $c_2$, have last coordinate $a_1$ and two centers, say $c_3$ and $c_4$, have last coordinate $a_2$. In this case, we have
\[\gamma_1^2+\gamma_2^2=||x_1-x_2||^2=||x_1-x_3||^2=\gamma_1^2+(a_1-a_2)^2+\gamma_3^2,\]
giving that $\gamma_2>\gamma_3$. Meanwhile,
\[\gamma_3^2+\gamma_4^2=||x_3-x_4||^2=||x_2-x_4||^2=\gamma_2^2+(a_1-a_2)^2+\gamma_4^2\]
gives that $\gamma_2<\gamma_3$, a contradiction. This completes the proof of~\Cref{bad_centers}.
\end{proof}

We proceed with the proof of~\Cref{simplex_stability}. Let $A_0$ be the set of points that are not in $\bigcup_{i\in[r]}A_i$. In summary, we have shown that one can find a partition $X=A_0\dot\cup A_1\dot\cup\dots\dot\cup A_r$ satisfying $|A_i|>n/r-(3r)^k\eps n$ for every $i\in[r]$ (see~\Cref{lustig}) and thus $|A_0|<r(3r)^k\eps n$, and pairwise orthogonal spheres $C_1,\dots,C_r\subseteq\RR^{2r+1}$ such that $A_i\subseteq C_i$ for all $i\in[r]$ (see~\Cref{lustig,creme}). Carrying out the argument with $\frac{\eps}{r(3r)^k}$ in place of $\eps$ from the beginning, we obtain $|A_0|<\eps n$ and $|A_i|>n/r-\eps n$ for each $i\in[r]$. It remains to check that $C_1,\dots,C_r$ satisfy items (1) and (2), or, in the case $r=3$, either items (1) and (2) or items (1') and (2').

Suppose first that $\sum_{i\in[r]}\dim\mathrm{Aff}(C_i)=2r+1$. Then we have $\dim\mathrm{Aff}(C_i)=2$ for each $i\in[r-1]$ and $\dim\mathrm{Aff}(C_r)=3$, so item (1) is fulfilled. Moreover, by~\Cref{good_centers} we know that $C_1,\dots,C_r$ share the same center, which together with~\Cref{spheres} implies that they have the same radius, thereby verifying item (2).

Suppose next that $\sum_{i\in[r]}\dim\mathrm{Aff}(C_i)=2r$ and at least $r-1$ among $C_1,\dots,C_r$ have the same center. In this case we have $\dim\mathrm{Aff}(C_i)=2$ for all $i\in[r]$. Since $C_1,\dots,C_r$ are all circles, we may, after relabeling if necessary, assume that $C_1,\dots,C_{r-1}$ share the same center denoted by $c$. Then it follows by~\Cref{spheres} that $C_1,\dots,C_{r-1}$ have the same radius. Furthermore, from~\Cref{spheres} one can deduce that the distance between $c$ and any point on $C_r$ is equal to this common radius. Recall $\sum_{i\in[r-1]}\dim\mathrm{Aff}(C_i)=2(r-1)$ and as demonstrated in the proof of~\Cref{bad_centers} we can assume that
\[\mathrm{Aff}(C_i)=c+\underbrace{\{0\}\times\dots\times\{0\}}_{2(i-1)\text{ times}}\times\RR^2\times\underbrace{\{0\}\times\dots\times\{0\}}_{2r+1-2i\text{ times}}\]
for each $i\in[r-1]$, and $\mathrm{Aff}(C_r)=c_r+\{0\}\times\dots\times\{0\}\times\RR^2\times\{0\}$, where $c_r$ denotes the center of $C_r$ which differs from $c$ only in the last coordinate. Let $\Gamma$ be the $2$-dimensional sphere in $c+\{0\}\times\dots\times\{0\}\times\RR^3$ centered at $c$ with the same radius as $C_1,\dots,C_{r-1}$. In particular, $\Gamma$ is orthogonal to $C_1,\dots,C_{r-1}$ and contains $C_r$. Therefore, in this case we can replace $C_r$ with $\Gamma$ (after which we relabel $\Gamma$ as $C_r$), and the resulting $C_1,\dots,C_r$ satisfy items (1) and (2). 

Lastly, suppose that $\sum_{i\in[r]}\dim\mathrm{Aff}(C_i)=2r$ and at most $r-2$ among $C_1,\dots,C_r$ have the same center. By~\Cref{bad_centers} this case only occurs when $r=3$. So we have three circles $C_1,C_2,C_3$ with pairwise distinct centers. Since the inter-circle distances are all equal by~\Cref{spheres}, we can conclude that $C_1,C_2,C_3$ satisfy items (1') and (2') in~\Cref{simplex_stability}. This completes the proof of~\Cref{simplex_stability}.
\end{proof}

\section{The number of regular $(k-1)$-simplices}
\label{sec:number}
In this section we prove our extremal result, namely,~\Cref{simplex_odd}. The following observation shall be used repeatedly in the proof.
\begin{lemma}
\label{obs}
Let $v\in\RR^d$, $d\geq3$. If $\Gamma\subseteq\RR^d$ is a $2$-dimensional sphere, then among any $n$ points on $\Gamma$, there are at most $O(n^{4/3})$ pairs of points that can form an equilateral triangle with $v$. If $\Gamma$ is a circle, then the number of such pairs is at most $n$.
\end{lemma}
\begin{proof}[Proof of~\Cref{obs}]
We partition the $n$ points on $\Gamma$ into classes $B_1,\dots,B_m$, such that every point in $B_j$ is at some distance $b_j\geq0$ from $v$. Note that any two points on $\Gamma$ that, together with $v$, form an equilateral triangle must lie in the same class $B_j$, with the pairwise distance equal to $b_j$. A classical result of Clarkson, Edelsbrunner, Guibas, Sharir, and Welzl~\cite{CEGSW90} states that the number of occurrences of any fixed distance among $n$ points on a $2$-dimensional sphere is at most $O(n^{4/3})$. Moreover, observe that any fixed distance can occur at most $n$ times among $n$ points on a circle. Therefore, the number of such pairs of points is at most $\sum_{j\in[m]}O(|B_j|^{4/3})\leq O(n^{4/3})$ if $\Gamma$ is a $2$-dimensional sphere, and at most $\sum_{j\in[m]}|B_j|\leq n$ if $\Gamma$ is a circle.
\end{proof}

\subsection{Excluding the wrong structure}
Let $r\geq k\geq3$ and suppose $X\subseteq\RR^{2r+1}$ is a set of $n$ points which spans the maximum number of regular $(k-1)$-simplices. In order to establish~\Cref{simplex_odd}, we shall first apply~\Cref{simplex_stability} to obtain a coarse structure of $X$, and then further examine this structure using the extremality of $X$. However, in the case when $r=3$ (and hence $k=3$),~\Cref{simplex_stability} provides two possible structures for $X$. Our first objective is to rule out one of them.

\begin{lemma}
\label{dimension7}
If $X\subseteq\RR^7$ is a set of $n$ points that spans the maximum number of equilateral triangles, where $n$ is sufficiently large, then $X$ must have the structure described by items (1) and (2) in~\Cref{simplex_stability}.
\end{lemma}
\begin{proof}[Proof of~\Cref{dimension7}]
Suppose $X$ has the structure described by items (1') and (2') in~\Cref{simplex_stability}, with some sufficiently small $\eps>0$. Namely, there are pairwise orthogonal circles $C_1,C_2,C_3\subseteq\RR^7$ with pairwise distinct centers, such that the distance between any two points on different circles is a constant $b>0$, and $|X\cap C_i|>n/3-\eps n$ for every $i\in[3]$. Let $c_i$ denote the center of $C_i$. Since $C_1,C_2,C_3$ are pairwise orthogonal and the projection of $C_i$ onto $\mathrm{Aff}(C_j)$ is $c_j$ for any distinct $i,j\in[3]$, by equipping $\RR^7$ with an orthonormal basis consisting of the orthonormal bases of the linear subspaces associated with $\mathrm{Aff}(C_1),\mathrm{Aff}(C_2),\mathrm{Aff}(C_3)$ and one additional basis vector, we can assume that
\[\mathrm{Aff}(C_i)=c_i+\underbrace{\{0\}\times\dots\times\{0\}}_{2(i-1)\text{ times}}\times\RR^2\times\underbrace{\{0\}\times\dots\times\{0\}}_{7-2i\text{ times}}\]
and $c_i=(x_1,\dots,x_6,a_i)\in\RR^7$ for each $i\in[3]$, where $x_1,\dots,x_6$ are fixed and $a_1,a_2,a_3$ are pairwise distinct.

Suppose $A_0=X\setminus(C_1\cup C_2\cup C_3)$ is non-empty and let $v\in A_0$. Let $N_1,N_2,N_3$ be the numbers of equilateral triangles containing
\begin{itemize}
    \item $v$ and at least another point from $A_0$;
    \item $v$ and two points on $C_i$ for some $i\in[3]$;
    \item $v$ and two points from different circles,
\end{itemize}
respectively. Since $|A_0|<\eps n$, $N_1<\eps n^2$. By~\Cref{obs} we have that $N_2\leq O(n)$. Furthermore, since the inter-circle distance is equal to $b$, it holds that $N_3=m_1m_2+m_1m_3+m_2m_3$, where $m_i$ denotes the number of points in $X\cap C_i$ that are at distance $b$ from $v$. We claim that there is at most one $i\in[3]$ with $m_i\geq3$, which in turn implies that $N_3\leq O(n)$. Indeed, suppose without loss of generality that $m_1,m_2\geq3$. Then the projections of $v$ onto $\mathrm{Aff}(C_1)$ and $\mathrm{Aff}(C_2)$ would be $c_1$ and $c_2$, namely we can write $v=(x_1,x_2,x_3,x_4,x_5',x_6',x_7')\in\RR^7$. Let $\gamma_1,\gamma_2,\gamma_3$ be the radii of $C_1,C_2,C_3$, respectively. Then
\[0=\left(||v-c_1||^2+\gamma_1^2\right)- \left(||v-c_2||^2+\gamma_2^2\right)=(x_7'-a_1)^2+\gamma_1^2-(x_7'-a_2)^2-\gamma_2^2\]
and
\[0=\left(\gamma_3^2+||c_3-c_1||^2+\gamma_1^2\right)- \left(\gamma_3^2+||c_3-c_2||^2+\gamma_2^2\right)=(a_3-a_1)^2+\gamma_1^2-(a_3-a_2)^2-\gamma_2^2,\]
yielding that $x_7'=a_3$. From $||v-c_1||^2+\gamma_1^2=b^2=\gamma_3^2+||c_3-c_1||^2+\gamma_1^2$ one can further deduce that $(x_5'-x_5)^2+(x_6'-x_6)^2=\gamma_3^2$, which implies $v\in C_3$, a contradiction to $v\in A_0$. This shows that $v$ is contained in $N_1+N_2+N_3<\eps n^2+O(n)$ equilateral triangles. Then, by replacing $v$ with any point on $C_1\setminus X$, and noting that such a point contributes at least $|X\cap C_2|\cdot|X\cap C_3|>n^2/16>N_1+N_2+N_3$ equilateral triangles, we obtain a new set of $n$ points spanning more equilateral triangles than $X$, contradicting the extremality of $X$. Therefore, we have $A_0=\emptyset$.

Since $X\subseteq C_1\cup C_2\cup C_3$, we can bound the number of equilateral triangles spanned by $X$ as follows. There are two types of equilateral triangles: those consisting of three points from pairwise distinct circles, and those containing at least two points from the same circle. The number of the first type triangles is at most $\prod_{i\in[3]}|X\cap C_i|\leq n^3/27$. Regarding the second type triangles, note that for any point $v\in X$, by~\Cref{obs} there are at most $|X\cap C_i|$ pairs of points from $C_i$ that can form an equilateral triangle with $v$. Therefore, the number of second type triangles is at most $n\cdot\sum_{i\in[3]}|X\cap C_i|=n^2$. This shows that $X$ spans at most $n^3/27+n^2$ equilateral triangles, which is strictly below the lower bound on $S^3_7(n)$ from~\Cref{simplex}, a contradiction.
\end{proof}

\subsection{Proof of~\Cref{simplex_odd}}
Let $r\geq k\geq3$ be fixed integers and let $\eps>0$ be sufficiently small. Let $X\subseteq\RR^{2r+1}$ be a set of $n$ points that spans $S^k_{2r+1}(n)$ regular $(k-1)$-simplices, where $n$ is sufficiently large with respect to $\eps$. By~\Cref{simplex_stability} and~\Cref{dimension7} we have a partition $X=A_0\dot\cup A_1\dot\cup\dots\dot\cup A_r$ with $|A_0|<\eps n$ and $|A_i|>n/r-\eps n$ for each $i\in[r]$, and pairwise orthogonal spheres $C_1,\dots,C_r\subseteq\RR^{2r+1}$ with $A_i\subseteq C_i$ for all $i\in[r]$, such that
\begin{enumerate}
    \item $\dim{C_1}=\dots=\dim{C_{r-1}}=1$ and $\dim{C_r}=2$;
    \item $C_1,\dots,C_r$ have the same center and the same radius.
\end{enumerate}
In particular, we can assume that $A_{0}=X\setminus(\bigcup_{i\in[r]}C_i)$.

Since $\mathrm{Aff}(C_1),\dots,\mathrm{Aff}(C_r)$ are pairwise orthogonal and intersect at the common center of $C_1,\dots,C_r$, we can further assume without loss of generality that the common center is the origin and \[\mathrm{Aff}(C_i)=\underbrace{\{0\}\times\dots\times\{0\}}_{2(i-1)\text{ times}}\times\RR^2\times\underbrace{\{0\}\times\dots\times\{0\}}_{2r+1-2i\text{ times}}\] for all $i\in[r-1]$ and $\mathrm{Aff}(C_r)=\{0\}\times\dots\times\{0\}\times\RR^3$. Let $\gamma>0$ denote the radius of $C_1,\dots,C_r$. For any $v\in\RR^{2r+1}$, denote $\mathcal{I}(v)\subseteq[r]$ the set of indices $i\in[r]$, such that the projection of $v$ onto $\mathrm{Aff}(C_i)$ is the center of $C_i$, i.e., the origin.

\begin{claim}
\label{garbage_points}
If $v\in A_0$, then
\begin{enumerate}[(i)]
    \item $||v||=\gamma$;
    \item $\mathcal{I}(v)\subseteq[r-1]$ and $|\mathcal{I}(v)|=r-2$;
    \item $v$ is at distance $\sqrt{2\gamma^2}$ to at least $2/3$ of the points in $X\cap C_r$.
\end{enumerate}
\end{claim}
\begin{proof}[Proof of~\Cref{garbage_points}]
We will show that if one of the claimed properties of $v$ does not hold, then $X$ is not an extremal set. Let $N_1,N_2,N_3$ denote the numbers of regular $(k-1)$-simplices that contain
\begin{itemize}
    \item $v$ and at least one other point from $A_0$;
    \item $v$ and at least two points on $C_i$ for some $i\in[r]$;
    \item $v$ and $k-1$ points from pairwise distinct spheres,
\end{itemize}
respectively. Then the number of regular $(k-1)$-simplices containing $v$ is at most $N_1+N_2+N_3$.

Since $|A_0|<\eps n$, we have that $N_1<\eps n^{k-1}$. For each $i\in[r]$, by~\Cref{obs} the number of pairs of points in $A_i$ that can form a regular $(k-1)$-simplex with $v$ is at most $O(n^{4/3})$, accordingly, $N_2\leq O(n^{k-5/3})$. If
\begin{equation}
\label{small}
N_3\leq\binom{r-1}{k-1}\left(\frac{n}{r}\right)^{k-1}-r^{-3k}n^{k-1},
\end{equation}
then any point $v'\in C_r\setminus X$ forms at least
\[\sum_{\mathcal{J}\in\binom{[r-1]}{k-1}}\prod_{j\in\mathcal{J}}|X\cap C_j|>\binom{r-1}{k-1}\left(\frac{n}{r}\right)^{k-1}-r^k\eps n^{k-1}>N_1+N_2+N_3\]
regular $(k-1)$-simplices with points in $X\setminus\{v\}$. Hence, replacing $v$ with $v'$ results in a new set of $n$ points spanning more regular $(k-1)$-simplices, contradicting the extremality of $X$. It remains to show that if one of items (i), (ii), and (iii) does not hold, then we have~\eqref{small}.

\textbf{Case 1: item (i) does not hold.}
Then the distance between $v$ and any point on $C_i$, $i\in\mathcal{I}(v)$, is not equal to $\sqrt{2\gamma^2}$. Since the distance between every two points from different spheres is $\sqrt{2\gamma^2}$, we have $N_3\leq\sum_{\mathcal{J}\in\binom{[r]\setminus\mathcal{I}(v)}{k-1}}\prod_{j\in\mathcal{J}}|X\cap C_j|$. Observe that for any $j\in[r-1]\setminus\mathcal{I}(v)$, there are at most $2$ points on $C_j$ at distance $\sqrt{2\gamma^2}$ from $v$, as otherwise the projection of $v$ onto $\mathrm{Aff}(C_j)$ would be the center of $C_j$, meaning that $j\in\mathcal{I}(v)$, a contradiction. This observation combined with the above upper bound on $N_3$ yields that $N_3\leq O(n)$.

\textbf{Case 2: item (ii) does not hold.}
Since $||v||=\gamma>0$, we have $|\mathcal{I}(v)|\leq r-1$, as otherwise $v$ would be the common center of $C_1,\dots,C_r$, i.e., the origin, a contradiction. Moreover, $|\mathcal{I}(v)|\neq r-1$. Indeed, if $\mathcal{I}(v)=[r]\setminus\{j\}$ for some $j\in[r]$, then $v\in\mathrm{Aff}(C_j)$ and $||v||=\gamma$, implying that $v\in C_j$, a contradiction. Now since $|\mathcal{I}(v)|\leq r-2$ and item (ii) does not hold, we must have either $r\in\mathcal{I}(v)$ or $|\mathcal{I}(v)|\leq r-3$. If $r\in\mathcal{I}(v)$, as there are at most $2$ points on $C_j$ at distance $\sqrt{2\gamma^2}$ from $v$ for each $j\in[r]\setminus\mathcal{I}(v)$, we have
\begin{align*}
N_3\leq\left(\sum_{\mathcal{J}\in\binom{\mathcal{I}(v)}{k-1}}\prod_{j\in\mathcal{J}}|X\cap C_j|\right)+O(n^{k-2})&\leq\binom{r-2}{k-1}\left(\frac{n}{r}+r\eps n\right)^{k-1}\\
&\leq\frac{(1+r^2\eps)^{k-1}}{r-1}\binom{r-1}{k-1}\left(\frac{n}{r}\right)^{k-1}\\
&\leq\binom{r-1}{k-1}\left(\frac{n}{r}\right)^{k-1}-r^{-3k}n^{k-1},
\end{align*}
yielding~\eqref{small}. Note that since the binomial coefficient here counts the number of $(k-1)$-element subsets, we default that $\binom{\alpha}{\beta}=0$ if $\alpha<\beta$. If $|\mathcal{I}(v)|\leq r-3$, then
\begin{align*}
N_3&\leq\left(\sum_{\mathcal{J}\in\binom{\mathcal{I}(v)}{k-1}}\prod_{j\in\mathcal{J}}|X\cap C_j|\right)+\left(\sum_{\mathcal{J}\in\binom{\mathcal{I}(v)}{k-2}}|X\cap C_r|\cdot\prod_{j\in\mathcal{J}}|X\cap C_j|\right)+O(n^{k-2})\\
&\leq\left(\binom{r-3}{k-1}+\binom{r-3}{k-2}\right)\left(\frac{n}{r}+r\eps n\right)^{k-1}\\
&=\binom{r-2}{k-1}\left(\frac{n}{r}+r\eps n\right)^{k-1}\leq\binom{r-1}{k-1}\left(\frac{n}{r}\right)^{k-1}-r^{-3k}n^{k-1},
\end{align*}
where in the last inequality we used the same estimate as in the above case.

\textbf{Case 3: item (iii) does not hold.}
In particular, $v$ has distance $\sqrt{2\gamma^2}$ to less than $2/3$ of the points in $X\cap C_r$. Then
\begin{align*}
N_3&\leq\left(\sum_{\mathcal{J}\in\binom{\mathcal{I}(v)}{k-1}}\prod_{j\in\mathcal{J}}|X\cap C_j|\right)+\left(\sum_{\mathcal{J}\in\binom{\mathcal{I}(v)}{k-2}}\frac{2|X\cap C_r|}{3}\cdot\prod_{j\in\mathcal{J}}|X\cap C_j|\right)+O(n^{k-2})\\
&\leq\left(\binom{r-2}{k-1}+\left(1-\frac{1}{3}\right)\binom{r-2}{k-2}\right)\left(\frac{n}{r}+r\eps n\right)^{k-1}\\
&=\left(\binom{r-1}{k-1}-\frac{1}{3}\binom{r-2}{k-2}\right)(1+r^2\eps)^{k-1}\left(\frac{n}{r}\right)^{k-1}\\
&\leq\binom{r-1}{k-1}\left(\frac{n}{r}\right)^{k-1}-\frac{1}{10}\binom{r-2}{k-2}\left(\frac{n}{r}\right)^{k-1}\leq\binom{r-1}{k-1}\left(\frac{n}{r}\right)^{k-1}-r^{-3k}n^{k-1},
\end{align*}
as desired. This completes the proof of~\Cref{garbage_points}.
\end{proof}

Next we shall define a new collection of spheres $\Gamma_1,\dots,\Gamma_r\subseteq\RR^{2r+1}$, each of dimension at most $2$, such that $X\subseteq\bigcup_{i\in[r]}\Gamma_i$.

If $A_0$ is empty, then we simply let $\Gamma_i=C_i$ for all $i\in[r]$. 

If $A_0\neq\emptyset$, then for any $v\in A_0$, by~\Cref{garbage_points} item (ii) we have $r\notin\mathcal{I}(v)$, i.e., the projection of $v$ onto $\mathrm{Aff}(C_r)$ is not the center of $C_r$. Meanwhile, by~\Cref{garbage_points} item (iii) $v$ is equidistant to at least $2/3$ of the points in $X\cap C_r$, implying that at least $2/3$ of the points in $X\cap C_r$ lie on a circle in the sphere $C_r$. Since there is at most one such circle in $C_r$, we conclude that the projections of the points in $A_0$ onto $\mathrm{Aff}(C_r)$ all lie on the line $L$ that goes through the center of this circle and is perpendicular to this circle. In particular, $L$ also goes through the center of $C_r$, which is the origin. Note that $L\subseteq\{0\}\times\dots\times\{0\}\times\RR^3$ is a $1$-dimensional linear subspace that is orthogonal to $\mathrm{Aff}(C_i)$ for every $i\in[r-1]$. Now let $\Gamma_r=C_r$, and for each $i\in[r-1]$, let $\Gamma_i$ be the $2$-dimensional sphere of radius $\gamma$ centered at the origin, lying in the $3$-dimensional subspace spanned by $\mathrm{Aff}(C_i)\cup L$. Since $C_i\subseteq\Gamma_i$ for all $i\in[r]$, we have $X\setminus A_0\subseteq\bigcup_{i\in[r]}\Gamma_i$. It suffices to check that $A_0\subseteq\bigcup_{i\in[r]}\Gamma_i$. For any $v\in A_0$, by~\Cref{garbage_points} item (ii) we have $\mathcal{I}(v)=[r-1]\setminus\{i\}$ for some $i\in[r-1]$, namely, $v$ lies in the subspace spanned by $\mathrm{Aff}(C_i)\cup\mathrm{Aff}(C_r)$. Since the projection of $v$ onto $\mathrm{Aff}(C_r)$ is contained in $L$, $v$ actually lies in the subspace spanned by $\mathrm{Aff}(C_i)\cup L$. Then by~\Cref{garbage_points} item (i) we have $v\in\Gamma_i$.

We can now bound the number of $(k-1)$-simplices spanned by $X$.

\begin{claim}
\label{odd_case}
$S^k_{2r+1}(n)\leq\binom{r}{k}\left(\frac{n}{r}\right)^k+O(n^{k-2/3})$.
\end{claim}
\begin{proof}[Proof of~\Cref{odd_case}]
Since $X\subseteq\bigcup_{i\in[r]}\Gamma_i$, there is a partition $X=V_1\dot\cup\dots\dot\cup V_r$ with $V_i\subseteq \Gamma_i$ for each $i\in[r]$. There are three types of simplices spanned by $X$:
\begin{itemize}
    \item regular $(k-1)$-simplices $\Delta_1$ with $|\Delta_1\cap V_i|\leq1$ for all $i\in[r]$;
    \item regular $(k-1)$-simplices $\Delta_2$ with $|\Delta_2\cap V_i|=2$ for some $i\in[r]$ and $|\Delta_2\cap V_i|\leq2$ for others;
    \item regular $(k-1)$-simplices $\Delta_3$ containing at least three vertices from some $V_i$.
\end{itemize}
The number of $\Delta_1$ is at most \[\sum_{\mathcal{I}\in\binom{[r]}{k}}\prod_{i\in\mathcal{I}}|V_i|\leq \binom{r}{k}\left(\frac{n}{r}\right)^k.\]
The number of $\Delta_2$ is at most $O(n^{k-2/3})$, since for each $i\in[r]$ and any $(k-2)$-element subset $T\subseteq X$, by~\Cref{obs} the number of pairs of points in $V_i$ that can form a regular $(k-1)$-simplex with $T$ is at most $O(n^{4/3})$.
Sharir, Sheffer, and Zahl~\cite{SSZ13} proved that any $n$ points in $\RR^3$ determine at most $O(n^{15/7})$ equilateral triangles. This implies that each $V_i$ contains at most $O(n^{15/7})$ triples that can contribute to a regular $(k-1)$-simplex, namely, the number of $\Delta_3$ is at most $O(n^{k-6/7})$. So the number of regular $(k-1)$-simplices spanned by $X$ is at most $\binom{r}{k}\left(\frac{n}{r}\right)^k+O(n^{k-2/3})$ as desired.
\end{proof}

Combining~\Cref{odd_case} with the lower bound given in~\Cref{simplex}, we have \[S^k_{2r+1}(n)=\binom{r}{k}\left(\frac{n}{r}\right)^k+\Theta(n^{k-2/3}),\] completing the proof of~\Cref{simplex_odd}.\qed

\section*{Acknowledgements}
F.~C.~Clemen is partially supported by a PIMS Postdoctoral Fellowship (PIMS-20260907-PDF). The authors thank Adrian Dumitrescu for valuable comments and an anonymous reader for pointing out a second possible configuration in $\RR^7$ in an earlier draft.

\end{document}